\documentclass[openacc]{rstransa}%%%%where rstrans is the template name

\newtheorem{theorem}{\bf Theorem}[section]
\newtheorem{lemma}[theorem]{\bf Lemma}%[section]
\newtheorem{condition}[theorem]{\bf Condition}%[section]
\newtheorem{corollary}[theorem]{\bf Corollary}%[section]
\newtheorem{definition}{\bf Definition}[section]
\newtheorem{remark}[definition]{\bf Remark}%[section]
\newtheorem{example}[definition]{\bf Example}%[section]
\newcommand{\mmt}[1]{\mathcal{#1}}
\newcommand{\mmts}[1]{\mathcal{#1}}
\newcommand{\mcover}[1]{\mathcal{#1}}
\newcommand{\mscen}[2]{\de{\mmts{#1},\mcover{#2}}}
\newcommand{\EM}[1]{\mathscr{#1}}
\newcommand{\R}{\mathbb{R}}

\newcommand{\de}[1]{\left(#1\right)}
\newcommand{\De}[1]{\left[#1\right]}
\newcommand{\DE}[1]{\left\{#1\right\}}
\newcommand{\abs}[1]{\left| #1\right|}

\newcommand{\eg}{\emph{e.g.}: }
\newcommand{\ie}{\emph{i.e.}: }

\newcommand{\OmSig}{\de{\Omega,\Sigma}}
\newcommand{\OmSigMu}{\de{\Omega,\Sigma,\mu}}
\newcommand{\Power}[1]{\mathcal{P}\de{#1}}
\newcommand{\sa}{$\sigma$-algebra }
\newcommand{\measf}[1]{\mathsf{#1}}

\usepackage{mathtools}
\usepackage{amssymb}
\usepackage{mathrsfs}

\begin{document}

%%%% Article title to be placed here
\title{On Measures and Measurements: a Fibre Bundle approach to Contextuality}

\author{%%%% Author details
Marcelo {Terra Cunha}$^{1}$}

%%%%%%%%% Insert author address here
\address{$^{1}$Departamento de Matem\'atica Aplicada, Instituto de Matem\'atica, Estat\'istica e Computa\c c\~ao Cient\'ifica, Universidade Estadual de Campinas, 13084-970, Campinas, S\~ao Paulo, Brazil}

%%%% Subject entries to be placed here %%%%
\subject{Probability Theory, Topology, Mathematical Foundations to Quantum Theory}

%%%% Keyword entries to be placed here %%%%

%%%% Insert corresponding author and its email address}
\corres{Marcelo Terra Cunha\\
\email{tcunha@unicamp.br}}

%%%% Abstract text to be placed here %%%%%%%%%%%%
\begin{abstract}
Contextuality is the failure of ``local'' probabilistic models to become global ones.
In this paper we introduce the notions of \emph{measurable fibre bundles},  \emph{probability fibre bundles}, and \emph{sample fibre bundle} which capture and make precise the former statement.
The central notions of contextuality are discussed under this formalism, examples worked out, and some new aspects pointed out.  
\end{abstract}
%%%%%%%%%%%%%%%%%%%%%%%%%%%

%%%%%%%%%% Insert the texts which can accomodate on firstpage in the tag "fmtext" %%%%%

\begin{fmtext}
\section{Outlook and Motivations}\label{sec:Intro}
%%%% 
Many Escher images, like \emph{Waterfall} or \emph{Ascending and Descending}\cite{Escher} illustrate vividly the problem of context\-uality.
Each piece of the image is totally consistent, the pieces even connect smoothly, however, when we try to glue them all together some ``obstruction'' appears.
Mathematically, the notion of gluing pieces together is also present.
Probably, the best known example are \emph{differential surfaces}, where each piece is identified with an open set of the Euclidean plane $\R^2$, but when the pieces are glued together, they can make very different objects like a sphere, a torus or even some objects which do not fit in  $\R^3$, like the Klein Bottle\cite{Topology}.
Another example are vector bundles.
If one uses a differentiable manifold (the $d$-dimensional analogue of diferentiable surfaces), $M$, as the basis space and attach to each point $p\in M$ a vector space, $V_p$, with some rule about how to connect the vectors of neighbour points, one gets another differentiable manifold, $E$, with points given by pairs $\de{p,v}$, with $p \in M$ and $v \in V_p$.
Well known to many physicists and essential for differential geometers  is the example of the tangent bundle, $TM$, when $V_p = T_pM$, the tangent space of $M$ on $p$, where Lagrangian mechanics is developed\cite{Arnold}.  
\end{fmtext}

%%%%%%%%%%%%%%% End of first page %%%%%%%%%%%%%%%%%%%%%

\maketitle

Probability theory, however, is developed over the concept of a single $\sigma$-algebra, where all the relevant events are defined.
This is very natural, classically, where given two events $A$ and $B$, it also makes sense to consider the event $A \cap B$, logically connected to the conjunction of the conditions for $A$ and $B$.
It inherently brings the notion of a ``global order'' and that images like the ones from Escher are really impossible in our world.
The notion of contextuality brings another status to those ``impossible figures''\cite{Penrose}.
Moreover, the notion of quantum contextuality\cite{KS} says that those impossible figures somehow appear in Nature.

In this contribution we develop the notions of \emph{measurable bundle} and \emph{probability bundle} as ``Escherian'' generalisations of the basic notions of measurable space and probability space, where Kolmogorovian probability theory is developed\cite{TaoProb}.
In other words, we set the basis for the development of a probability theory where topology plays a crucial \textit{r\^ole}.

Despite being inspired by quantum theory, the notions developed here are totally independent of it.
The examples worked out here are the most simple, yet surprising ones, and do not demand any knowledge of quantum theory.
After reviewing (in sec.~\ref{sec:ClassProb}) the basic notions of probability theory which we are generalising, we introduce the central notions of contextuality and the central objects of our contribution in sec.~\ref{sec:Defs}.
The fibre bundle approach to contextuality is developed in sec.~\ref{sec:FBapproach}, with examples and the translations of Fine-Abramsky-Brandeburger theorem for characterisation of contextuality and Budroni-Morchio results on scenarios with no contextuality to this language.
The notion of contextuality subscenarios guide the sec.~\ref{sec:Subscenarios}, emphasising even more the connection of contextuality and topology.
Connections to other approaches are discussed in sec.~\ref{sec:Others}, while some more speculative points are raised in sec.~\ref{sec:Disc} before the closing in sec.~\ref{sec:Conc}.

%the specific connection to quantum theory is being worked out by the community

%We adhere in this proposal to the (operationally inspired) notion that \emph{measurements} are fundamental objects.

\section{Classical Probability Theory in a nutshell}\label{sec:ClassProb}

\subsection{The basic notions}\label{sub:ProbNotions}
A famous dictum  attributed to M. Kac say that probability theory is measure theory with a soul.
This section is to fix concepts and notations for sample spaces, measurable spaces and probability spaces.
For a deeper introduction, we recommend ref.~\cite{TaoProb}, among many other textbooks.
The central notion is
\begin{definition}\label{def:SA}
 Given a set $S$, a \emph{\sa }over $S$, denoted $\Sigma$, is a collection of subsets of $S$ with $\emptyset \in \Sigma$, closed under complementation and countable unions. 
\end{definition}
\begin{remark}\label{rmk:Boole}
 The examples to be raised will use finite sets $S$.
 However, some of the questions and the central target of generalising usual probability theory demand the more general definition. 
\end{remark}
\begin{definition}\label{def:MeasurableSpace}
 A \emph{measurable space} is a pair $\OmSig$, where $\Omega$ is a set and $\Sigma$ a \sa on it.
 The set $\Omega$ is called the \emph{sample space}.
\end{definition}
The next step is to measure the sets in $\Sigma$.
This demands the notion of a \emph{measure}.
We will jump straight to the definition of \emph{probability measure}:
\begin{definition}\label{def:Prob}
 A \emph{probability measure}, or simply a \emph{probability}, on a measurable space $\OmSig$ is a function $\mu: \Sigma \longrightarrow \De{0,1}$ such that
  \begin{itemize}
   \item $\mu\de{\emptyset} = 0$;
%   \item $\mu\de{A} \geq 0$ for all $A \in \Sigma$;
   \item If $A_i$ are pairwise disjoint sets in $\Sigma$, then\footnote{Again, this property is demanded on all countable unions and countable sums.} $\displaystyle{\mu\de{\bigcup _i A_i} = \sum_i \mu\de{A_i}}$;
   \item $\mu\de{\Omega} = 1$.
  \end{itemize}
 The triple $\de{\Omega,\Sigma,\mu}$ is called a \emph{probability space}.
\end{definition}

%%% Include coarsening of a \sa
Another important notion in what follows is that if we fix a sample space $\Omega$, the $\sigma$-algebras on it obey a partial order coming from inclusion:
\begin{definition}\label{def:CG}
 Given two $\sigma$-algebras, $\Sigma$ and $\Sigma'$, over the same set we say that $\Sigma$ is a \emph{coarse graining} of $\Sigma'$, or that $\Sigma'$ \emph{refines} $\Sigma$ if $\Sigma \subseteq \Sigma'$.
\end{definition}
Intuitively, this means that any atom of $\Sigma$ can be obtained as a (countable) union of atoms of $\Sigma'$.

\subsection{The simplest case and its geometry}\label{sub:ProbSimplex}
%%%%%%%%%%%%%%%%%
Once one is concerned with a finite collection of (classical) random variables, $\mmts{X} = \DE{X_i}$, each of them taking values in finite sets $S_i$, there is a minimal and sufficient (Boolean) \sa to hold such situation.   
For only one variable, taking values in the finite set $S$, the powerset, $\Power{S}$, plays this \textit{r\^ole}.
In other words, one can identify $S$ with the sample space $\Omega$ and use the measurable space $\de{S,\Power{S}}$.
For the finite set of variables, 
\begin{equation}\label{eq:classsig}
 \Sigma = \prod _{\mmts{X}} \Power{S_i}.
\end{equation} 
Here, $\prod$ is taken as the \sa product, analogous to the set-theoretical Cartesian product.
Interestingly, the simple set-theoretical identification $\prod \Power{S_i} \equiv \Power{\prod S_i}$ brings something that will have deep consequences when contextuality comes to play: classically, we could change this set of random variables for just one random variable, $X$, taking values in $\prod S_i$.
Anticipating some notions (see Definition \ref{def:JointRealisation}), this means that all (classical) random variables are jointly realisable.  

%% Probability simplex %%%
Whenever one wants to attach a probability measure to a \sa $\Power{S}$, it is necessary and sufficient to attach a non-negative value $p\de{s}$ for each $s\in S$, such that $\sum_S p\de{s} = 1$.
This brings a very basic, concrete, and geometrical notion:
\begin{definition}\label{def:ProbSimplex}
 Given a finite set, $S$, the canonical vectors of $\mathbb{R}^{\abs{S}}$ are the \emph{vertices of the probability simplex}.
 The \emph{probability simplex} is the convex hull of its vertices.
\end{definition}
Indeed, this means that the usual picture of a segment, a triangle, a tetrahedron, \ldots, generalises and that the only important feature of $S$ is its cardinality. 
Each point in the simplex works as a probabilistic model for the set of variables studied and this is a bijection.

\subsection{An Ontology for Classical Probability Theory}\label{sub:ClassOntology}
%%% Include self interpretation of Kolmogorov's axioms ``ontological completion''
Classical probability theory, as introduced by Kolmogorov, allows for a very natural ontology.
It is, by no means, necessary, and the Occam's razor must be the reason for its (complete?) absence in textbooks.
The sample space, $\Omega$, can be considered as an ontological space, where every question has an answer, or, equivalently, every property is well defined.
Those questions or properties are given by random variables, each of them assuming well defined values for each point $\omega \in \Omega$.
Their randomness only enters in the game under the idea that $\omega$ is unknown (or, ``hidden'').
To some sense, the epistemological space is given by $\Sigma$, or by some coarsening  $\Sigma' \subset \Sigma$, such that all random variables defined in the model has measurable sets attached to each possible value.

As already stated, such ontology is unnecessary.
However, it can be considered as \emph{the} classical trait to be consistent with such ontology.
It is under this paradigm that space shuttles are sent to Mars, that GPS localises mobiles up to small uncertainties, and even classical statistical mechanics was developed\footnote{Actually, it was such a completion that Einstein, Podolsky, and Rosen were looking for in their criticism to quantum theory\cite{EPR} and this is what Bell rules out with his inequalities\cite{Bell}.}.
Contextuality appears whenever it fails.
This is usually from where adjectives like weird, counterintuitive, or paradoxical come from.  

%%%%%%%%%%%%%%%%%%%%%%%%%%%%%%%%%%%%%%%%%%%%%%%%%%%%
\section{Central Definitions}\label{sec:Defs}

\subsection{Measurements}\label{sub:Measuremt}
The most primitive notion of measurement is that it gives results.
This is what we will use in order to characterise them:
\begin{definition}\label{def:Measurmt}
 A \emph{measurement}, $\mmt{M}$, is a (countable\footnote{For measurements with uncountable possible results some adaptations will be necessary, but in this contribution we want to keep things simple.}) set of labels for the possible results. 
\end{definition}

\begin{definition}\label{def:Realisation}
 Given a measurement, $\mmt{M}$, a \emph{realisation} in a measurable space $\de{\Omega,\Sigma}$ is a partition of $\Omega$ into elements of $\Sigma$ subordinated to $\mmt{M}$, \ie with elements indexed by the measurement outcomes.
 %In order to keep  things simple, we will only consider finite sets in this contribution.
\end{definition}
\begin{remark}\label{rmk:RealisationExistence}
 At this level, all (finite) measurements are realisable.
 Just consider $\Omega = \mmt{M}$, as sets, and $\Sigma = \Power{\Omega}$.
\end{remark}
\begin{remark}\label{rmk:MeasurableFunctions}
 Equivalently, one can define a realisation of $\mmt{M}$ in $\OmSig$ as a measurable function $\measf{m}: \OmSig \rightarrow \de{\mmt{M},\Power{\mmt{M}}}$.
\end{remark}
\begin{remark}\label{rmk:RealisationExclusiveness}
 Note that by this definition, whenever a measurement is done, one and only one of the outcomes appear.
 This is consistent with the notion of \emph{exclusive} outcomes.
\end{remark}
Whenever the measurable space is upgraded to a probability space, a realisation of $\mmt{M}$ defines probabilities for all its outcomes:
\begin{definition}\label{def:Measure}
 Given a measurement $\mmt{M}$ with a realisation in $\de{\Omega,\Sigma}$, if $\de{\Omega,\Sigma,\mu}$ is a probability space, then $\mu$ is a (probability) \emph{measure} for $\mmt{M}$.
\end{definition}
\begin{remark}\label{rmk:pushfwrd}
 If we go along with remark \ref{rmk:MeasurableFunctions}, we can also see that  pushing forward $\mu$ along $\measf{m}$ gives a (probability) measure on $\de{\mmt{M},\Power{\mmt{M}}}$. 
\end{remark}

Before jumping into contextuality, we need to workout the notions related to joint measurability. %It would be nice to include some references on JM
\begin{definition}\label{def:JointRealisation}
 Given two measurements, $\mmt{M}$ and $\mmt{N}$, the \emph{joint measurement} is given by $\mmt{M}\wedge\mmt{N} = \mmt{M}\times\mmt{N}$ (as sets).
 A \emph{joint realisation} in a measurable space $\OmSig$ is a partition of $\Omega$ into elements of $\Sigma$ subordinated to $\mmt{M}\wedge\mmt{N}$.
 If $\OmSigMu$ is a probability space, then $\mu$ is a \emph{joint (probability) measure} for $\mmt{M}$ and $\mmt{N}$.
\end{definition}
\begin{remark}\label{rmk:ProductMeasurement}
 If $\mmt{M}$ is realised in $\OmSig$ and $\mmt{N}$ is realised in $\de{\Omega',\Sigma'}$, then $\mmt{M}\wedge\mmt{N}$ is realised in $\OmSig\times\de{\Omega',\Sigma'}$.
 This antecipates that the central notion of contextuality does not appear when considering only two measurements.
 This very simple observation has deep consequences: contextuality can never appear on trivial topologies\cite{Budronietal}.
\end{remark} 
%%% Interesting discussion for infinite measurements: are there two realisable measurements which can not be jointly realised?
%%% Simpler point: is there a case where this product realisation failure?
%
%
\begin{definition}\label{def:CG}
 Given a set $M$ and a partition $P$ of $M$ into disjoint nonempty subsets, the measurement $\mmt{P}$ associated to $P$ is a \emph{coarse graining} of the measurement $\mmt{M}$ associated to $M$.
\end{definition}
\begin{definition}\label{def:MarginalMeasurements}
 Given a joint measurement $\mmt{M}\wedge\mmt{N}$, the measurements $\mmt{M}$ and $\mmt{N}$ are called its \emph{marginal measurements}.
\end{definition}
\begin{lemma}\label{lemma:MarginalMeasurements}
 Given a joint measurement $\mmt{M}\wedge\mmt{N}$, the marginal measurements can be identified with coarse grainings of the joint measurement.
\end{lemma}
\begin{proof}
 Given the Cartesian product $\mmt{M}\times\mmt{N}$, the partition $m\times\mmt{N}$ for $\mmt{M}\times\mmt{N}$ is naturally identified with the set $\mmt{M}$, by $m\times\mmt{N} \longmapsto m$.
 Analogously, for the partition $\mmt{M}\times n$.  
\end{proof}
This leads to another very important relation between measurements:
\begin{definition}\label{def:Compat}
 Measurements $\mmt{M}$ and $\mmt{N}$ are \emph{compatible} iff there exist a common coarse graining, from which both can be obtained as marginals.   
\end{definition}
\begin{remark}
Given two measurements and no other restriction, one can always define the joint measurement.
The strength of the notion of compatibility only shows up when we work on some nontrivial \emph{contextuality scenario}, which brings us to the next subsection.
\end{remark}
%%%%
\subsection{Contextuality}\label{sub:Context}
%%%% 
Now we set the stage for the questions we are treating.
The central ingredient is a collection, $\mmts{X}$, of possible measurements, $\mmt{M}_k$, some of them can be compatible, while other not.
%Setting these conditions defines a \emph{contextuality scenario}:
\begin{definition}\label{def:Context}
  Given a collection of measurements, $\mmts{X}$, a \emph{context} for the measurement $\mmt{M}\in\mmts{X}$ is a subset $C \subset \mmts{X}$ with $\mmt{M}\in C$ in which all measurements are compatible.
  A \emph{maximal context} for $\mmt{M} \in \mmts{X}$, $C$, is a context that can not be enlarged by including elements of $\mmts{X}$.
  In other words, if $C$ is a maximal context in $\mmts{X}$, for any $\mmt{N}\in\mmts{X}\setminus C$ there is some  $\mmt{M}' \in C$ incompatible with $\mmt{N}$.
\end{definition}
As already stated, in this construction, compatibility is not an inherent property of a set of measurements.
This demands the central notion of \emph{scenarios}:
\begin{definition}\label{def:Scenario}
  A \emph{measurement scenario}, or a \emph{contextuality scenario}, or a \emph{compatibility scenario} (which we will call simply a \emph{scenario}) is given by a pair $\mscen{X}{C}$, where $\mmts{X}$ is a collection of measurements and $\mcover{C}$ a compatibility cover for $\mmts{X}$. 
  A \emph{compatibility cover}, $\mcover{C}$, for the measurement collection $\mmts{X}$ is a collection of maximal contexts in $\mmts{X}$ such that each $\mmt{M} \in \mmts{X}$ belongs to at least one context in $\mcover{C}$.
\end{definition}
\begin{remark}
  The choice of demanding maximal contexts for the cover is technical.
  One great advantage is economy when describing scenarios.
  The other possible choice would be to impose that whenever $C \in \mcover{C}$, than $C'\in \mcover{C}$ for every $C'\subset C$.
  The later has the topological advantage of making simplicial complexes explicit.
\end{remark}

%%%% Empirical model %%%
In usual experiments, for each context every measurements in the context can be performed and the results sampled allowing the inference of some probability distribution for the joint results of the measurements in each context.
\begin{definition}\label{def:EM}
 Given a scenario $\mscen{X}{C}$, an \emph{empirical model}, $\EM{E}$, is the association of each context $C \in \mcover{C}$ to a probability distribution of the results of the measurements in $C$, \ie $p_C: \prod_{\mmt{M}\in C}\mmt{M} \rightarrow \mathbb{R}$, with $p_c\de{\mathbf{m}} \geq 0$ and $\sum p_c\de{\mathbf{m}}  = 1$.
\end{definition}

%%%% Non-disturbing %%%
Given an empirical model (also called a \emph{behaviour}), it also defines probability distributions for subsets of the (maximal) contexts, via marginalisation.
If $S\subset{C}$, then $p_S^C: \prod_{\mmt{M}\in S}\mmt{M} \rightarrow \mathbb{R}$ is given by
\begin{equation}\label{eq:marginal}
 p_S^C\de{\mathbf{s}} = \sum_{\mathbf{m}_{|S}=\mathbf{s}} p_C\de{\mathbf{m}},
\end{equation}
where $\mathbf{m}_{|S}=\mathbf{s}$ means that the restriction\footnote{Some readers may consider more elegant to consider a projection $\pi: \mathbf{m} \mapsto \mathbf{s}$ and the sum runs over the fibre over $\mathbf{s}$, $\pi^{-1}\de{\mathbf{s}}$.
The same projection can also be used to say that the marginal $p^C_S$ is the result of pushing forward $p_C$ through $\pi$.} of the outcomes $\mathbf{m}$ for the measurements in $C$ to those on $S$ gives the respective outcomes $\mathbf{s}$.

One important condition in those problems is the following:
\begin{condition}\label{cond:MC}
  If two contexts, $C$ and $C'$, overlap, the \emph{marginal condition} demands $\displaystyle{p_{C\cap C'}^C = p_{C\cap C'}^{C'}}$.
\end{condition}
Whenever this condition holds, we may drop the symbol from the context out and work with $p_{C\cap C'}$.
In other words, this is the necessary condition for a probability distribution to be defined on $C \cap C'$, and, consequently, in all of its subsets. 
\begin{definition}\label{def:NDist}
  An empirical model, $\EM{E}$, in a scenario, $\mscen{X}{C}$, is \emph{non-disturbing} if for every $C,C' \in \mcover{C}$ the marginal condition \ref{cond:MC} holds. 
\end{definition} 

%%%% Non contextual %%%
After defining a non-contextual model, let us present two different definitions for an empirical model to be non contextual.
These definitions are shown equivalent by the Fine-Abramsky-Brandenburger Theorem \cite{Fine,AB}.

%%%% Non contextual model %%%
\begin{definition}\label{def:NCmodel}
 Given a scenario, $\mscen{X}{C}$, a \emph{non-contextual model} is a collection of probability distributions $\displaystyle{p_{\mmt{M}}^{\Lambda}: \mmt{M} \times \Lambda \rightarrow \mathbb{R}}$, with $\displaystyle{p_{\mmt{M}}^{\Lambda}\de{m,\lambda} \geq 0}$  and $\displaystyle{\sum_{m\in\mmt{M}} p_{\mmt{M}}^{\Lambda}\de{m,\lambda} = p^{\Lambda}\de{\lambda}}$, for all $\mmt{M}\in\mmts{X}$ and $\lambda\in \Lambda$, such that for every $C\in\mcover{C}$,
   \begin{equation}\label{eq:NCModel}
     p_C\de{\mathbf{m}} = \sum_{\lambda \in \Lambda} \prod_{\mmt{M}\in C} p_{\mmt{M}}^{\Lambda} \de{m,\lambda}.
   \end{equation}
\end{definition}
\begin{remark}\label{rmk:NCModel}
   The existence of a non-contextual model can be read as the existence of a probability space $\de{\Lambda, \Upsilon, p^{\Lambda}}$ from which one can obtain independent joint probabilities for each random variable in the scenario and lambda.
   Following the discussion in subsection \ref{sec:ClassProb}.\ref{sub:ClassOntology}, this is consistent with the idea that $\lambda$ determines\footnote{Explicitly, given a model \eqref{eq:NCModel}, fixing $\lambda$ only determines conditional probabilities $p_{\mmt{M}}\de{m|\lambda}$, however, as again those are classical probabilities, this means that we could refine more and obtain $\de{\Lambda', \Upsilon', p^{\Lambda'}}$ such that all those conditional probabilities be deterministic.} the outcome of each random variable and correlations are established since one cannot access the variable $\lambda$.
   This is essentially the reasoning for such additional (latent, in the language of \emph{causal structures}\cite{Pearl}) variable $\lambda$ be usually referred to as \emph{hidden}.  
\end{remark}

\begin{definition}\label{def:NCbyHV}
 An empirical model, $\EM{E}$, is \emph{non contextual by model} if it can be obtained from a non-contextual model.
\end{definition}
%%%% Non contextual Marginal %%%
\begin{definition}\label{def:NCbyMarg}
 An emprical model, $\EM{E}$, is \emph{non contextual by marginals} if there is a joint probability distribution $p_{\mmts{X}}: \prod_{\mmt{M}\in\mmts{X}} \rightarrow \mathbb{R}$ such that for all $C \in \mcover{C}$, $p_C = p_C^{\mmts{X}}$.
\end{definition}
As already anticipated:
\begin{theorem}[Fine, Abramsky, Brandenburger]\label{thm:FAB}
 An empirical model, $\EM{E}$, is non contextual by model if, and only if, it is non contextual by marginals.
\end{theorem}
%%%%%%%%%%%%%%%%%%%%%%%%%
%%%%%%%%%%%%%%%%%%%%%%%%%Pensar se carece de explicar melhor aqui,
%%%%%%%%%%%%%%%%%%%%%%%%%ou ainda como tornar a prova do teorema FAB simples
%%%%%%%%%%%%%%%%%%%%%%%%%ainda que antes de introduzir os fibrados.
%%%%%%%%%%%%%%%%%%%%%%%%%
%%%%%%%%%%%%%%%%%%%%%%%%%
Given this equivalence, whenever the distinction is irrelevant we will just call those empirical models non contextual.
A proof can be made very constructive: given a non-contextual model, by multiplying all $p_{\mmt{M}}^{\Lambda}$ and then marginalising on $\Lambda$, one obtains the collective distribution.
On the other way around, the collective distribution can always be made as a convex combination of deterministic distributions, which assign deterministic values for each random variable; one simply needs to use the variables of the convex combination to act as $\lambda$. 
We hope our work will shed some new light on this important result, by stressing the r\^ole played by topology in it.

%%%%
\subsection{Measurable Fibre Bundles}\label{sub:MeasFB}
%%%%
Now we start to glue concepts.
Given a scenario, $\mscen{X}{C}$, by construction for each $C \in \mcover{C}$, one can choose a measurable space $\de{\Omega^C,\Sigma^C}$ which realises the joint measurement $\displaystyle{\bigwedge _{\mmt{M} \in C} \mmt{M}}$.
Each $\mmt{M}\in C$ gives a partition of $\Sigma^C$ and defines a coarsening of $\Sigma^C$, which we will denote $\Sigma^C_{\mmt{M}}$. 
Whenever $C \cap C' \neq \emptyset$, the joint measurement of all elements of $C \cap C'$ will define the coarsenings $\Sigma^C_{C\cap C'}$ and $\Sigma^{C'}_{C\cap C'}$.
The labelling given by such coarsening defines a bijection  $\Sigma^C_{C\cap C'} \leftrightarrow \Sigma^{C'}_{C\cap C'}$.
By identifying the sets related by such bijections, we have an abstract \sa for each $\mmt{M} \in \mmts{X}$ given by the equivalence class of all $\Sigma^C_{\mmt{M}}$, for all $C \ni \mmt{M}$, which we denote by $\Sigma_{\mmt{M}}$.
This construction prepares the following:
\begin{definition}\label{def:MeasurableBundles}
  Given a scenario $\mscen{X}{C}$, a \emph{Measurable Fibre Bundle} is the attachment of a measurable space $\de{\Omega^C,\Sigma^C}$ to each context.
  Whenever $C \cap C' \neq \emptyset$, the identification $\Sigma^C_{C\cap C'} \leftrightarrow \Sigma^{C'}_{C\cap C'}$ is considered, leading to the \emph{fibre over the measurement $\mmt{M}$}, $\Sigma_{\mmt{M}}$.
  The \emph{projection}, $\pi$, is defined only for the sets $S \in \Sigma^C$ associated to a set in $\Sigma^C_{\mmt{M}}$.  Naturally, in those cases $\pi\de{S} = \mmt{M}$.
\end{definition}
It is quite important to note that this is done on the level of the \sa but not necessarily on the level of the sample space.
Let us discuss a situation where one can go up until the level of sample spaces.

%%% Define Sample Fibre Bundle and show that if it is done non-contextually, then the bundle is trivial! %%%
\begin{definition}\label{def:SampleBundle}
  A measurable fibre bundle, $\de{\Omega^C,\Sigma^C}$ for a scenario, $\mscen{X}{C}$ is a \emph{Sample Fibre Bundle} if the bijections $\Sigma^C_{C\cap C'} \leftrightarrow \Sigma^{C'}_{C\cap C'}$ can consistently be obtained from relations $\Omega^C \leftrightarrow \Omega^{C'}$.
\end{definition}
\begin{remark}
  It is important to note that those relations are not necessarily functions, since multiple associations are allowed in both directions, however, their images must generate partitions subordinated to $C\cap C'$.
\end{remark}
%%% Workout the example from a $\de{\Omega^{\mmt{M}},\Sigma^{\mmt{M}}}$ for each $\mmt{M}$ and then collect together the products for each $C$ %%%
\begin{example}\label{ex:Product}
  Given a scenario, $\mscen{X}{C}$, for each $\mmt{M} \in \mmts{X}$ choose a measurable space, $\de{\Omega^{\mmt{M}},\Sigma^{\mmt{M}}}$.
  For each $C \in \mcover{C}$, take 
     \begin{equation}\label{eq:SampleBundle}
        \de{\Omega^C,\Sigma^C} = \prod_{\mmt{M}\in C}\de{\Omega^{\mmt{M}},\Sigma^{\mmt{M}}}.
     \end{equation}
   Whenever $\mmt{M} \in C \cap C'$, the identification $\Sigma^C_{C\cap C'} \leftrightarrow \Sigma^{C'}_{C\cap C'}$ must fix the $\Sigma^{\mmt{M}}$ component in the product.
   This can be naturally extended to $\Omega^{\mmt{M}}$, also acting as the identity map on it.
   The Measurable Bundle just constructed is a Sample Bundle.  
\end{example}
%%% Discuss this is locally a product, but in this case, extensible to a complete product %%%
The example above explores one very important characteristic of fibre bundles: it is piecewise a product\footnote{The best mathematical word is \emph{locally}, but we will not use it here in order to avoid any misunderstanding related to the locality concept of Bell.}.
However, there is a little bit more: each $\de{\Omega^C,\Sigma^C}$ is actually a restriction of a global product:
\begin{equation}\label{eq:global}
  \de{\Omega,\Sigma} = \prod_{\mmt{M} \in \mmts{X}} \de{\Omega^{\mmt{M}},\Sigma^{\mmt{M}}}.
\end{equation} 
In this sense, those fibre bundles are trivial.
A nontrivial example comes from the $n-$cycle scenario, bearing in mind the celebrated M\"obius strip:
\begin{example}\label{ex:Moebius}
 Consider a dychothomic  $n$-cycle scenario, $\mscen{X}{C}$, where $\mmts{X} = \DE{\mmt{M}_0,\mmt{M}_1,\ldots,\mmt{M}_{n-1}}$, each $\mmt{M}_i$ a binary set, and $\mcover{C} = \DE{C_i}_{i=0,\ldots,n-1}$, with $C_i = \DE{\mmt{M}_i,\mmt{M}_{i+1}}$, and addition understood modulo $n$.
 For each $\mmt{M}_i$, take $\Omega^{\mmt{M}_i} = \De{-1,1}$ and Borel sets as $\Sigma^{\mmt{M}_i}$.
 For each $C_i$, take $\de{\Omega^{C_i},\Sigma^{C_i}} = \de{\Omega^{\mmt{M}_i},\Sigma^{\mmt{M}_i}} \times \de{\Omega^{\mmt{M}_{i+1}},\Sigma^{\mmt{M}_{i+1}}}$ and,  as identifications $\Omega_i \leftrightarrow \Omega_{i+1}$ take identity functions $\omega_i \mapsto \omega_{i+1}$ for all, but one values of $i$, for which, $\omega _i \mapsto -\omega_{i+1}$ is used.
 Since those identifications extend to Borel sets, here we have explictly constructed a Sample Bundle.
\end{example} 
In contrast to example \ref{ex:Product}, in example \ref{ex:Moebius} we have nontrivial bundles. 
\begin{remark}\label{rmk:sectionsMoebius}
It is important to recognise that for all values of $c \neq 0$, the piecewise constant function $i \mapsto c$ fails to generate a global function.
On the other hand, the constant function $i \mapsto 0$ is globally defined and this has interesting consequences to which we shall come back latter on.
\end{remark}
%
%\begin{remark}\label{rmk:commoncoarsening}
%  Is this discussion, we can consider the sets of $\Sigma^C$  labeled by outcomes of all the measurements in $C$.
%  Each subset of $C$ corresponds to a coarsening of $\Sigma^C$ and, under the identification above, $C\cap C'$ can be associated to the finest common coarsening of $\Sigma^C$ and $\Sigma^{C'}$.
%%% Is this right??? %%%
%\end{remark} 

%Whenever $C \cap C' \neq \emptyset$, the elements of each $\mmt{M} \in C \cap C'$ label two partitions, one in $\Sigma$ and one in $\Sigma'$.
%This defines a bijective relation between the sets of these partitions.
%It also defines a common coarsening for $\Sigma$ and $\Sigma'$, which we can. 
%Since $\Sigma$ and $\Sigma'$ are both \sa, when we consider the intersections of the sets associated to each element of each partition, we obtain $\tilde{\Sigma}$ and $\tilde{Sigma'}$    

%%%%
\subsection{Probability Fibre Bundles}\label{sub:ProbFB}
Now we will include probabilities and empirical models in our discussion.
The measurable bundles of subsection \ref{sub:MeasFB} will receive probability measures in each fibre and the marginal condition will receive a new interpretation.

The construction specialises the notion of measurable bundle.
Given a scenario $\mscen{X}{C}$, we may attach a probability space to each context: $\de{\Omega^C,\Sigma^C,\mu^C}$.
Whenever $C\cap C' \neq \emptyset$, we want not only to have the identification $\Sigma^C_{C\cap C'} \leftrightarrow \Sigma^{C'}_{C\cap C'}$, but also to define a probability measure on it.
This is exactly the marginal condition \ref{cond:MC} applied to the restrictions $\mu^{C}_{C\cap C'}$ and $\mu^{C'}_{C\cap C'}$.
\begin{definition}\label{def:ProbFB}
  Given a scenario $\mscen{X}{C}$, a \emph{Probability Fibre Bundle} is the attachment of a probability space $\de{\Omega^C,\Sigma^C,\mu^C}$ to each context.
  Whenever $C\cap C' \neq \emptyset$, the identification $\Sigma^C_{C\cap C'} \leftrightarrow \Sigma^{C'}_{C\cap C'}$ is considered and the marginal condition $\mu^C_{C\cap C'} = \mu^{C'}_{C\cap C'}$ is demanded, leading to the fibre over the measurement $\mmt{M}$, $\de{\Sigma_{\mmt{M}},\mu_{\mmt{M}}}$.
  The projection, $\pi$, is defined only for the sets $S\in \Sigma^{C}$ associated to a set in $\Sigma^C_{\mmt{M}}$.
  Naturally, in those cases, $\pi\de{S,\mu_{\mmt{M}}\de{S}} = \mmt{M}$.
\end{definition}

The definition of Probability Fibre Bundles, \ref{def:ProbFB}, should be compared to the definition of Empirical Model, \ref{def:EM}.
The latter gives the necessary data to the former, whenever the marginal condition holds for all overlapping contexts.
The problem can be stated as this: given a non-disturbing empirical model on a measurement scenario $\mscen{X}{C}$, one can construct a probability fibre bundle over the same scenario.
Two questions appear: can we upgrade such a bundle to a sample fibre bundle?
If so, is this bundle trivial?

%%% Bell model for spin 1/2 -- Let us write in this language!! -- probably not here...

%Simplicial complexes? --- Point out, including the Elsewhere
To close this section, let us point out another aspect.
In subsection \ref{sec:ClassProb}.\ref{sub:ProbSimplex} we introduced the probability simplex.
For scenarios with finite contexts of finite measurements, each context generates a probability simplex.
Whenever we condition on any (subset) of the measurements in the context, we obtain a smaller simplex, related to the remaining measurements.
If we consider the intersection of two contexts, there is no meaning in conditioning on incompatible variables.
However, the marginals are meaningful and they represent coarse grainings which should, then, be identifiable.
What is the geometry/topology behind such identifications and, specially, coming from such identifications over a given scenario? 

%%%%%%%%%%%%%%%%%%%%%%%%%%%%%%%%%%%%%%%%%%%%%%%%%%%%
\section{Fibre Bundle approach to Contextuality}\label{sec:FBapproach}
Now we have everything set and can discuss in more details how contextuality manifests in those objects.
The central result is the translation of Fine-Abramsky-Brandeburguer Theorem into the question of triviality of a Probability Fibre Bundle.
This brings to the arena questions on extensions and obstructions, \ie on possibilities and impossibilities.
It also shows the centrality of the so far ignored notion of subscenarios, which we shall discuss on section \ref{sec:Subscenarios}.
%%%%
% Key results
%%%% Not all measurable bundles originate sample bundles %%%% IS this true?? Example!!!
\subsection{An Example-Oriented Discussion}\label{sub:examples}
The first thing to be understood is the generalisation of the situation shown in subsection \ref{sec:ClassProb}.\ref{sub:ProbSimplex}.
There, for a finite collection of finite sets, the sample space could be identified with the atoms of the \sa of the problem.
This can be put together with example \ref{ex:Product} to show that any finite scenario allows for a trivial sample bundle.
Let us discuss another example of non-trivial sample bundle.
\begin{subequations}
\begin{example}[Hollow Triangle]\label{ex:HollowTriangle}
 Consider the scenario $\mscen{X}{C}$ given by 
 \begin{align}\label{eq:HollowTriangleScen}
   \begin{split}
      \mmts{X} &= \DE{\mmt{M}_a,\mmt{M}_b, \mmt{M}_c},\\ 
      \mcover{C} &= \DE{\DE{\mmt{M}_a,\mmt{M}_b}, \DE{\mmt{M}_b,\mmt{M}_c},\DE{\mmt{M}_c,\mmt{M}_a}},
   \end{split}
  \end{align} 
 with $\mmt{M}_a = \DE{\uparrow,\downarrow}$, $\mmt{M}_b = \DE{0,1}$, and $\mmt{M}_c = \DE{g,r}$.
 For each context take the minimal measurable space 
 \begin{align}\label{eq:HollowTriangleSA}
   \begin{split}
     \de{\Omega^{ab},\Sigma^{ab}} &= \de{\mmt{M}_a\times\mmt{M}_b, \Power{\mmt{M}_a\times\mmt{M}_b}},\\
     \de{\Omega^{bc},\Sigma^{bc}} &= \de{\mmt{M}_b\times\mmt{M}_c, \Power{\mmt{M}_b\times\mmt{M}_c}},\\
     \de{\Omega^{ca},\Sigma^{ca}} &= \de{\mmt{M}_c\times\mmt{M}_a, \Power{\mmt{M}_c\times\mmt{M}_a}}.
   \end{split}
  \end{align} 
  Now we need the coarse graining in order to define the identifications.
  Here we will introduce contextuality, which can be seen as a twist (topologically saying): for each context, the first variable will have the ``natural'' coarse graining, while the second will be ``inverted''.
  Explicitly:
 \begin{align}\label{eq:HollowTriangleSA}
   \begin{split}
     \Omega^{ab}_a &= \DE{\uparrow = \DE{\de{\uparrow,0},\de{\uparrow,1}},\downarrow = \DE{\de{\downarrow,0},\de{\downarrow,1}}}, \\
     \Omega^{ab}_b &= \DE{0 = \DE{\de{\uparrow,1},\de{\downarrow,1}},1 = \DE{\de{\uparrow,0},\de{\downarrow,0}}}, \\
     \Omega^{bc}_b &= \DE{0 = \DE{\de{0,g},\de{0,r}},1 = \DE{\de{1,g},\de{1,r}}}, \\
     \Omega^{bc}_c &= \DE{g = \DE{\de{0,r},\de{1,r}},r = \DE{\de{0,g},\de{1,g}}}, \\
     \Omega^{ca}_c &= \DE{g = \DE{\de{g,\uparrow},\de{g,\downarrow}},r = \DE{\de{r,\uparrow},\de{r,\downarrow}}}, \\
     \Omega^{ca}_a &= \DE{\uparrow = \DE{\de{g,\downarrow},\de{r,\downarrow}},\downarrow = \DE{\de{g,\uparrow},\de{r,\uparrow}}}. 
   \end{split}
  \end{align} 
  This defines a non-trivial sample bundle.
  We shall keep this notation that makes contextuality and the twist explicit, by using such minimal construction where $\Sigma^C = \Power{\Omega^C}$ and the coarse grainings are projections or ``twisted projections''; but in order to not consider this as artificially made, the reader should also consider the equivalent case where $\Omega^C_S$ is a binary set with labels unrelated to the elements of $\Omega^C$.
  
  To see how contextuality is deeply related to such bundle, we should interpret what a choice like $\de{\uparrow, 1,g}$ would mean in terms of the measurements.
  In the context $\DE{\mmt{M}_a,\mmt{M}_b}$ this implies outcomes $\de{\uparrow, 0}$, while in the context $\DE{\mmt{M}_b,\mmt{M}_c}$ this implies $\de{1, r}$, which already shows the contextuality of $\mmt{M}_b$. 
  Accordingly, for the context $\DE{\mmt{M}_c,\mmt{M}_a}$ this implies $\de{g, \downarrow}$, making clear that those ``global assignments'' which try to define values for all measurements can exhibit contextuality.
  Generalising the choice made, any assignment $\de{a,b,c}$ would imply opposite answers for each measurement on its two contexts. 
\end{example} 
\end{subequations}
This example generalises for any size, $n$, of the cycle.
As we shall see later on, there is a parity issue decisive in order to distinguish \emph{intrinsic} contextuality from \emph{removable} contextuality.

A very beautiful phenomenon shows up when building probability bundles over sample bundles like the one in example \ref{ex:HollowTriangle}.
In those examples, contextuality is manifest in its stronger case: disturbance.
Suppose one can characterise that a ``state'' $\de{a,b,c}$ was prepared. 
Then, when measuring $\mmt{M}_a$, the result $a$ would imply the context $\de{\mmt{M}_a,\mmt{M}_b}$, while the result $\bar{a}$ would imply the context $\de{\mmt{M}_c,\mmt{M}_a}$.
\begin{example}\label{ex:SParable}
  Consider the Hollow Triangle scenario of example \ref{ex:HollowTriangle} and the following empirical model on it:
   \begin{align}\label{eq:HollowTriangleScen}
   \begin{split}
     p_{ab}\de{\uparrow,1} &= \frac12 = p_{ab}\de{\downarrow,0},\\ 
     p_{bc}\de{0,r} &= \frac12 = p_{bc}\de{1,g},\\ 
     p_{ca}\de{g,\downarrow} &= \frac12 = p_{ca}\de{r,\uparrow}, 
   \end{split}
   \end{align}
with all other probabilities null.
First, by checking the marginals, see that this empirical model is nondisturbing.
Now, one can easily check that this model can be considered as a balanced   convex combination of the two ``global assignments'': $\de{\uparrow, 0, g}$ and $\de{\downarrow,1,r}$.
Other three equally interesting sample bundle realisations of (contextual) probability bundles are from the other combinations of $\de{a,b,c}$ with $\de{\bar{a},\bar{b},\bar{c}}$.
Topologically, those non-disturbing empirical models are related to the null constant function of remark \ref{rmk:sectionsMoebius}.
\end{example} 
We have seen two examples of sample bundles, one trivial, one not, and some probability bundles which realise given empirical models.

\subsection{Making Topology (more) Explicit}\label{sub:Topo}
The central question now appears: given a scenario, can all empirical models be obtained as trivial probability bundles?
The interesting answer is: \emph{only the non-contextual ones!}
That is how Fine-Abramsky-Brandenburger Theorem \ref{thm:FAB} translates here:
%%%% Product sample bundles mean non-contextuality %%%%
\begin{theorem}[Fibre Budle version of Fine-Abramsky-Brandenburger]\label{thm:BundleFAB}
   Given a scenario, $\mscen{X}{C}$, an empirical model can be realised as a trivial probability bundle iff the model is non contextual. 
\end{theorem}
\begin{proof}
   If the empirical model can be realised as a trivial probability bundle, then the sample bundle can be extended to a product bundle and the model is non contextual by marginals.
   
   If the empirical model, $\EM{E}$, is non contextual there is a global probability distribution for the measurements in $\mmts{X}$ and this is an empirical model $\tilde{\EM{E}}$ in the scenario $\mscen{X}{X}$, giving a trivial bundle  whose restriction to $\mscen{X}{C}$ gives a trivial probability bundle realising $\EM{E}$.
\end{proof}

In the proof above we already used one instance of a very important topological result: not all basis (in our case, scenarios) can support non-trivial bundles\footnote{In the proof we explicitly use that a trivial scenario $\mscen{X}{X}$ supports only trivial bundles.}.
This explains and extends a lot the remark \ref{rmk:ProductMeasurement}, and put in a topologically broader context the result on ref.~\cite{Budronietal}.
This deserves a little more discussion.
In a very lousy way, we could say that trivial basis can only generate trivial bundles.
To make it precise, one should define what a trivial basis mean.
One very good example  of ``trivial basis''' is any contractible hypergraph.
Since a cycle always make a (hyper-)graph non-contractible, in graph language, the trivial basis/scenarios are trees.
\begin{corollary}[Budroni-Morchio]\label{cor:Costantino}%%% Check not only the reference, but if their result is really this.
  If a scenario $\mscen{X}{C}$ is free of cycles, than it is also free of contextuality.
\end{corollary}   
% Include topological bundle results, like contractible base implies triviality, interpreting as contextuality and relating to literature (e.g. Costantino).
Topologically, cycles allow for non-trivial homology\footnote{They are closed ``curves'' which are not boundaries.}.
In refs.\cite{RSS,ASBKLM,OTR,Caru18}, the authors point out the influence of cohomology in contextuality and nonlocality.
We hope this bundle approach will reforce this relation.
% Discuss also the notion of obstruction for the general problem of making extensions, interpreted here as conditions for contextuality.
As characterised in Thm.~\ref{thm:BundleFAB}, the central question in contextuality is about extending a probability fibre bundle into a product one.
In topology, it is usual to work with \emph{obstructions}, \eg a nontrivial homology class implies the nontriviality of some bundle and forbids such extension.
In this sense, a nontrivial first homology group says that a scenario allows for contextuality, while a nontrivial cohomology class is a way of witnessing contextuality\cite{RSS,ASBKLM,OTR,Caru18}.

% Start the same discussion from the n-cycle example, with intrinsic/removable contextuality.

% Call the question for the topological invariant for contextuality (cite Abramsky, Rui, and Shane's homology), opening the question: what about higher homology groups?  

%%%%
%\subsection{Examples}\label{sub:examples}
%Start from the very broad notions of where contextuality appears and how to put it on this framework

%Workout simple examples, in quantum theory or not, where beautiful topology appears

\subsection{Some comments on models}\label{sub:Models}
 It is interesting to come back to the example of the Hollow Triangle, \ref{ex:HollowTriangle}, and its generalisations for the $n$-cycle.
 We naturally should contrast it with the product bundle, example \ref{ex:Product}, build on the same scenario.
 A very natural question is: are there other interesting sample bundles in this scenario?
 The answer is: essentially no!
 One could guess: the product bundle has no twist while in example \ref{ex:HollowTriangle} we made three twists; what about one or two twists?
 Interestingly, these cases correspond to relabelings of the previous two cases: if we take the product bundle but ``flip'' one of its variables, it will generate a bundle with two twists, which, however, is something we can call \emph{removable contextuality}.
 Such a bundle is not explicitly a product, but it is isomorphic to one, so it can not support contextuality.
 Analogously, if we flip on variable from the example \ref{ex:HollowTriangle}, we will ``eliminate two flips'', but we still get a non-trivial bundle, as one should guess.
 
%%%% Nontrivial sample bundles are ``realistic models'' in nontrivial topologies %%%%
Another interesting point to call is that while trivial sample bundles allow for ``realistic models'' and classical interpretations, those non-trivial sample bundles can be considered as ``realistic'' models for ``classical'' interpretations on topologically richer ontological spaces\footnote{Again, no one has to adhere to such ``ontology'', neither the author. But is seems very beautiful and deserves more study, at least from the mathematical viewpoint.}.

%%%%%%%%%%%%%%%%%%%%%%%%%%%%%%%%%%%%%%%%%%%%%%%%%%%%
\section{Contextuality Subscenarios}\label{sec:Subscenarios}
%%%%
When we interpret noncontextuality as the possibility of describing an empirical model using a trivial probability bundle, \ie as an extension problem, another concept pops up: what about other extensions?
In order to make this question more precise, we need a partial order relation for scenarios:
\begin{definition}\label{def:Subscenario}
 A scenario $\mscen{X'}{C'}$ is a \emph{subscenario} of $\mscen{X}{C}$ if $\mmt{X'} \subseteq \mmt{X}$ and $\mcover{C'} \leq \mcover{C}$, where this last condition means that for every $C' \in \mcover{C'}$ there is $C \in \mcover{C}$ such that $C' \subseteq C$.
\end{definition}
We will denote $\mscen{X'}{C'} \preceq \mscen{X}{C}$ when $\mscen{X'}{C'}$ be a subscenario of $\mscen{X}{C}$.
\begin{remark}\label{rmk:subscen}
The classical scenarios are $\mscen{X}{X}$, for any $\mmts{X}$, and for every scenario it is true that $\mscen{X}{C} \preceq \mscen{X}{X}$.
\end{remark}
There are two canonical ways of creating subscenarios, and any subscenario can be obtained using any or both of them combined.
%% Induced subscenarios
\begin{definition}\label{def:IndSubscen}
  $\mscen{X'}{C'} \preceq \mscen{X}{C}$ is an \emph{induced subscenario} if $\mcover{C'} = \mcover{C} \sqcap \Power{\mmts{X'}}$, where this last symbol is used in the sense that $C'\in \mcover{C} \sqcap \Power{\mmts{X'}}$ iff one of two situations happen: either $C' \in \mcover{C} \cap \Power{\mmts{X'}}$ or $C' = \tilde{C}_{|X'}$ for some $\tilde{C} \in \mcover{C}$. 
\end{definition}
%% Context-Restricted subscenarios
\begin{definition}\label{def:RestrSubscen}
 $\mscen{X'}{C'} \preceq \mscen{X}{C}$ is a \emph{(context-)restricted subscenario} if $\mmts{X'} = \mmts{X}$. 
\end{definition}
By remark \ref{rmk:subscen}, every scenario is a context-restricted subscenario of the classical scenario for measurements $\mmts{X}$.

Definition \ref{def:NCbyMarg} can now be extended:
\begin{definition}\label{def:extend}
 Given $\mscen{X'}{C'} \preceq \mscen{X}{C}$, an empirical model $\EM{E'}$ in the scenario $\mscen{X'}{C'}$ \emph{extends} to the scenario $\mscen{X}{C}$ if there is an empirical model $\EM{E}$ in $\mscen{X}{C}$ such that for all $C' \in \mcover{C'}, p_{C'}^{\EM{E}} = p_{C'}^{\EM{E'}}$. 
\end{definition}
Simply putting those definitions together one gets:
\begin{theorem}\label{thm:NCasClassicalExtension}
 An empirical model $\EM{E}$ in $\mscen{X}{C}$ is noncontextual iff it extends to the classical scenario $\mscen{X}{X}$.
\end{theorem}
%%% Make a chain os subscenarios and ask where it breaks!
A beautiful mathematical structure appears when we build a sequence of nested scenarios, and ask about the possible extensions:
\begin{definition}\label{def:NestScen}
  A \emph{sequence of scenarios} is an ordered collection of scenarios $\DE{{\mscen{X}{C}_i}}$ where ${\mscen{X}{C}}_i \preceq {\mscen{X}{C}}_{i+1}$.
%  A sequence of scenarios starting from $\mscen{X}{C}$ \emph{finishes} if it reaches $\mscen{X}{X}$.
  A sequence of scenarios is \emph{complete} if in every step ${\mscen{X}{C}}_i$ is a proper subscenario of ${\mscen{X}{C}}_{i+1}$ such that there is no proper subscenario in between them.
\end{definition}
This notion of sequence of scenarios allows for growing sets of questions, but a central question comes for sequences of the form
\begin{equation}\label{eq:NestScen}
 {\mscen{X}{C}} \prec \ldots \prec {\mscen{X}{C}}_i \prec \ldots \prec {\mscen{X}{X}},
\end{equation}
\ie a strictly increasing sequence that finishes on the classical scenario for the same set of measurements. A very simple result is:
\begin{theorem}\label{thm:extensions}
 If an empirical model in $\mscen{X}{C}$ is noncontextual, then it extends to all elements of the sequence \eqref{eq:NestScen}.
\end{theorem}
The beautiful question comes: if an empirical model in  $\mscen{X}{C}$ is contextual, which extensions are allowed in a sequence  \eqref{eq:NestScen} and which are not?
In other words: at which steps contextuality really makes its presence?
We close this section with a short example.
\begin{example}\label{ex:SubscenCycle}
 Let us take the $5$-cycle as the starting scenario: $\mmts{X} = \DE{\mmt{X}_i}_{i=1\ldots 5}$, $\mcover{C} =  \DE{\DE{\mmt{X}_i, \mmt{X}_{i+1}}}_{i=1\ldots 5}$.
 This scenario supports a bundle like the one in the Hollow Triangle example, \ref{ex:SParable}.
 We can consider the coarser scenario with the same measurements and $\mcover{C} =  \DE{\DE{\mmt{X}_{i-1},\mmt{X}_i, \mmt{X}_{i+1}}}_{i=1\ldots 5}$ and we see that the model extends to it. It also extends to the scenario given by $\mcover{C} =  \DE{\DE{\mmt{X}_{i-1},\mmt{X}_i, \mmt{X}_{i+1},\mmt{X}_{i+2}}}_{i=1\ldots 5}$.
 The proof of theorem \ref{thm:extensions} can be applied to other extensions, like here, to say that  the empirical model extends to  any scenario in between the cycle and this four-element-context cover.
 Since no proper subscenario fits in between this last one and the classical, we can say that it is this step that allows for contextuality.
\end{example}
 Other interesting consequences of the notion of subscenarios will be presented elsewhere\cite{NonMonogamy}. 

%%%%%%%%%%%%%%%%%%%%%%%%%%%%%%%%%%%%%%%%%%%%%%%%%%%%
\section{Connection to other approaches and previous literature}\label{sec:Others}
%%%%
% Contextuality by default
Naturally, the here proposed approach to contextuality is not isolated from other proposals.
The starting point of describing each measurement in each of its contexts has some similarity with \emph{Contextuality by Default} (CbD) \cite{CbD}.
The difference appears when we demand the marginal condition \ref{cond:MC} to hold, while CbD relax it, in order to be able of treating empirical data which usually do not obey such condition strictly.
Probably the notion of \emph{extended contextuality}\cite{ExtCont} be the link in between the two approaches, but some interesting notions can come up from the difference among original contexts and extended ones.

% (Pre-)Sheaf approach
The central \textit{r\^ole} played by topology is also present in the (Pre-)Sheaf approach\cite{AB}, to which we have already made reference.
% Shane & Rui on extendability
The notion of extendability and its links to topology had been discussed previously\cite{RS} and we should always mention that the idea of gluing probability spaces and ask for its (non-)triviality was already present in ref.~\cite{Vorob}.
 
All those contemporary notions have some debt with the \emph{graph approach to contextuality}\cite{CSW} and to its variations\cite{ATCspringer}.
\section{Discussions and Future Developments}\label{sec:Disc}
Many points of this approach must be developed elsewhere and also many questions are still open.

One very natural thing is to consider \emph{continuous measurement scenarios}.
A natural example comes from spin (or polarisation) measurements on a qbit. 
There, the Bloch (or Poincar\'e) sphere serve as a parameter space for the measurements of the scenario; however, the measurement cover  would be made of singletons with no transition function defined.
Clearly, we need a good definition of a \emph{connection}, in order to associate elements of ``neighbour'' fibres, even when they do not belong to the same context.
A very good geometrical question is: after defining such a connection, how its curvature relates to contextuality?

A topologically interesting question comes to higher order homology.
The first homology group is deeply connected with the existence of  nontrivial cycles\footnote{By construction, cycles which are not boundaries are usually understood as folding some ``hole'', from where the nontriviality comes.}.
And nontrivial first homology groups is sufficient for the existence of nontrivial probability bundles which realise contextual empirical models, including examples coming from quantum theory.
Is it possible to have a measurement scenario with trivial first homology group, but nontrivial higher order homology groups, like a sphere, which also support nontrivial probability bundles?
Moreover, are those contextual models obtainable from quantum theory?
Or is it true that quantum contextuality really depends on the non-triviality of the first homology group\cite{Caru17}?

%Another urgent question is how to define the (first) homology class for a given probability bundle, and how to relate it with non-contextuality inequalities. 

%%% Connect the covers of the bundles to the possibility of nontrivial subscenarios
The notion of subscenarios brings with it the idea of covers of a given bundle.
Example \ref{ex:SubscenCycle} has discussed a collection of different subscenarios, but all with the same topology, except by the last one, where contextuality shows up.
In this sense we identify that ``closing the holes'' is the essential step in order to reduce the possibilities of contextuality.
The universal cover of a probability bundle will be deeply related to the notion of extended contextuality and this also seems to deserve more studies.

%%% If no answer is found for the ``simplex glueing question'' raised, comment it here as open. 
Another empty avenue to be travelled comes from the geometry of the simplest example: the probability simplex of subsec.~\ref{sec:ClassProb}.\ref{sub:ProbSimplex}.
We see that the marginal condition \ref{cond:MC} implies some interesting identifications.
Moreover, whenever nontrivial probability bundles appear, the global picture becomes richer and richer.
To have a good description of some examples is a nice target for the near future.

\section{Conclusion}\label{sec:Conc}
In this paper we introduced a fibre bundle approach to contextuality.
Actually, three kinds of bundles play interesting \textit{r\^oles} in this approach: measurable bundles, probability bundles, and sample bundles.
Mathematically, this proposal represents a merging of topological notions with probability theory.
From the viewpoint of contextuality, it is another approach, with many similarities with the (pre-)sheaf approach, but possibly more comprehensible to many physicists.
Naturally, another approach can allow for new interpretations and new insights. 
Some interesting questions were raised in this pages and we believe that their answers will help us figuring out many interesting aspects of contextuality.

\vskip6pt
\enlargethispage{20pt}

%\ethics{Insert ethics statement here if applicable.}

%\dataccess{Insert details of how to access any supporting data here.}

%\aucontribute{For manuscripts with two or more authors, insert details of the authors’ contributions here. This should take the form: 'AB caried out the experiments. CD performed the data analysis. EF conceived of and designed the study, and drafted the manuscript All auhtors read and approved the manuscript'.}

\competing{The author declares that he has no competing interests.}

\funding{This work received support from the Brazilian agency CNPq, also from the Brazilian National Institute for Science and Technology on Quantum Information, and finally from Purdue Winer Memorial Foundation.}

\ack{This work also took benefit from collaborations and discussions with many colleagues. I am specially indebted with Samson Abramsky, Barbara Amaral, Ad\'an Cabello, Ehtibar Dhzafarov, Leonardo Guerini, Shane Mansfield,  Ricardo Mosna, Roberto Imbuzeiro Oliveira, Paulo Ruffino, and Rui Soares Barbosa for different reasons. 
	J.S. Bach has probably helped with the quality of the text. 
	Unfortunately, he could not correct my own mistakes.}

\disclaimer{Important to say that the connection between Escher paintings and contextuality has long been used by Oxfordians like Shane Mansfield and Samson Abramsky, even during a time when I believed to have been the first to make such connection.}

%%%%%%%%%% Insert bibliography here %%%%%%%%%%%%%%

\end{document}